\documentclass[11pt]{nyjm}
\usepackage{xypic}

\title[Canonical representations in positive characteristic]{On the canonical representation of curves in positive characteristic}

\author{Ruthi Hortsch}  
\address{Department of Mathematics,
Massachusetts Institute of Technology,
77 Massachusetts Avenue,
Cambridge, MA 02139 } 
\email{rhortsch@math.mit.edu}   
\thanks{This paper was researched and written for an NSF REU under the mentorship of Bryden Cais. The help of Stephen DeBacker was also essential for this project. Additional thanks goes to E. Brooks, B. Conrad, S. Glasman, D. Moreland, A. Kumar, B. Poonen, K. Smith, L. Spice, and P. Srinivasan, as well as to the referee, who made several helpful suggestions.} 



\newtheorem{theorem}{Theorem}
\newtheorem{lemma}[theorem]{Lemma}
\newtheorem{prop}[theorem]{Proposition}
\newtheorem{cor}[theorem]{Corollary}

\numberwithin{equation}{section}
\numberwithin{theorem}{section}

\newcommand{\HD}{H^0(X, \Omega^1_{X/k})}
\newcommand{\HdR}{H^1_{dR} (X/k)}
\newcommand{\HF}{H^1(X, \mathcal{O}_X)}
\renewcommand{\div}{\operatorname{div}}
\newcommand{\Aut}{\operatorname{Aut}}
\newcommand{\Sym}{\operatorname{Sym}}
\newcommand{\F}{{\bf F}}

\newcommand{\SL}{\operatorname{SL}}

\newcommand{\vHdR}{\check{H}^1_{dR}(\mathcal{U})}
\newcommand{\isommap}{\stackrel{\sim}{\to}}

\begin{document}

\begin{abstract}  
Given a smooth curve, the canonical representation of its automorphism group is the space of global regular 1-forms as a representation of the automorphism group of the curve. In this paper, we study an explicit set of curves in positive characteristic with irreducible canonical representation whose genus is unbounded. Additionally, we study the implications this has for the de Rham hypercohomology as a representation of the automorphism group.
\end{abstract}

\maketitle
\setcounter{tocdepth}{1}
\tableofcontents

\section{Introduction}

\subsection{Motivation}
Consider a smooth projective curve $X$ of genus $g$ over an algebraically closed field $k$. The {\em canonical representation} of $G:=\Aut(X)$ is the $g$-dimensional $k$-representation of $G$ induced by the natural action of $G$ on the space of global regular 1-forms on $X$. In characteristic zero, the canonical representation has been determined in the work of Chevalley and Weil, responding to a query by Hecke (see \cite{ChWe}, \cite{Weil}), and these methods extend to positive characteristic if the characteristic of $k$ does not divide $|G|$. If the characteristic divides $|G|$, Kani and Nakajima have separately resolved the case where $X \rightarrow X/G$ is tamely ramified (see \cite{Kani}, \cite{Nak1}, \cite{Nak2}), while Valentini and Madan resolved the case where $G$ is cyclic of a $p$-power degree (see \cite{ValMad}) (a result that Nakajima later generalized to any invertible $G$-sheaf, see \cite{Nak3}). Most recently, K\"ock has generalized these results to the situation where the cover is weakly ramified, i.e. when all the higher ramification groups vanish at all points (see \cite{Ko}).

In this paper, we focus on the natural question of when the canonical representation is irreducible. If it is and $k$ has characteristic zero, then $g^2 \le |G|$. Together with the Hurwitz bound $|G| \le 84(g-1)$, this implies that $g \le 82$. This situation is studied in \S19 of \cite{Bre}, which includes a list of all possible genera and automorphism groups for which the canonical representation is irreducible. The Klein quartic $X(7)$ is a particularly important example of this phenomenon in characteristic zero (see \cite{Elk}).

The situation is quite different in positive characteristic. The Hurwitz bound continues to hold for characteristic $p>0$, provided that $2 \le g<p$, with the sole exception of the curve $y^2=x^p-x$ (see \cite{Roq}). However, in the general case, the bounds on irreducible representations are weaker, and thus we get no upper bound independent of $|G|$ on the genera of curves with irreducible canonical representation. This presents the possibility of having curves with arbitrarily large genus and an irreducible canonical representation. This paper shows that this holds for the projective smooth curve corresponding to $y^2=x^p-x$, which has genus $(p-1)/2$ (this curve is also studied over finite fields in \cite{GJK}). The question remains open whether for fixed $p$ there exist curves over $k$ with irreducible canonical representation of arbitrary large dimension.

\subsection{Summary of Results}

Let $k$ be an algebraically closed field of characteristic $p > 2$, and let $X$ be the unique smooth projective curve over k with affine equation $y^2=x^p-x$. Set $G:=\Aut_k(X)$, which is a degree $2$ central extension of $\operatorname{PGL}_2(\F_p)$ (discussed in \cite{Roq} \S5). Explicitly, let $\widetilde{G} \subseteq GL_2(\F_p) \times \F_{p^2}^\times$ be the subgroup of elements $(\sigma, u_\sigma)$ with $\det \sigma = u_\sigma^2$. Then $\F_p^\times$ acts diagonally on $\widetilde{G}$ via \[\lambda (\sigma, u_{\sigma})= \left( \left(\begin{smallmatrix}
\lambda & 0\\
0 & \lambda \\
\end{smallmatrix}\right) \sigma, \lambda^{(p+1)/2} u_\sigma \right),\] and $G= \F_p^\times \backslash \widetilde{G}$. We will consistently represent elements of $G$ by chosen lifts in $\widetilde{G}$. Then if $\sigma = \left(\begin{smallmatrix}
a & b\\
c & d\\
\end{smallmatrix}\right) \in GL_2(\F_p)$, and $u_\sigma$ is a choice of square root of $\det \sigma$, we can write the action of $G$ on the affine coordinates of $X$ as
\begin{equation*}
(x,y) \mapsto \left(\frac{ax+b}{cx+d}, u_{\sigma}\frac{y}{(cx+d)^{(p+1)/2}}\right).
\end{equation*}

In this paper, we show that $\HD$ is irreducible as a $k[G]$-module. We do this by considering a small index subgroup $H$ of $G$ that is a quotient of $\SL_2(\F_p)$. Noting that $\HD$ has dimension $g=(p-1)/2$ as a $k$-vector space, we get that the restriction of the canonical representation to this subgroup is the $(g-1)^{st}$ symmetric power of the standard representation of $\SL_2(\F_p)$ on $k^2$, which is well-known to be irreducible. It is not difficult to then describe $\HD$ as a $k[G]$-submodule of the induced representation of $\Sym^{g-1}(k^2)$.

Note that since $G$ is not a cyclic $p$-group, this is beyond the results of Valentini and Madan mentioned earlier. Furthermore, the cover $X \rightarrow X/G$ is wildly (but not weakly) ramified: the point at infinity and $\{ (x, 0) \in X : x \in \F_p \}$ are an orbit of $X$ under $G$, and since this orbit has $p+1$ elements, each is ramified of degree $2p(p-1)$. The previous work of Kani and Nakajima applies only where $X \rightarrow X/G$ is tamely ramified, and thus not to this case.

In the second half of this paper, we consider the de Rham hypercohomology as a $k[G]$-module. In characteristic $p$, Maschke's Theorem does not apply if $p$ divides $|G|$ and it is possible to find representations that are reducible but not semisimple. If a nonzero representation is not the direct sum of two proper subrepresentations, call it {\em indecomposable}. The de Rham hypercohomology of any curve $C$ is clearly reducible, since $H^0(C, \Omega_{C/k}^1)$ embeds as a $k[\Aut(C)]$-submodule, but we show that in the case of the curve $X$, the de Rham cohomology is indecomposable as a $k[G]$-module. 

It has been previously shown that the crystalline cohomology of $X$ tensored with ${\bf Q}_p$ is irreducible (see \cite{Katz}, Proposition 7.2). Knowing that the crystalline cohomology is irreducible as a $k[G]$-module, it is not necessarily true that $\HdR$ should be indecomposable as a $k[G]$-module, since its $G$-invariant subrepresentations do not lift to the crystalline cohomology. Thus we do not necessarily expect $\HdR$ to be indecomposable as a $k[G]$-module, especially since it is an unusual (and distinctly positive characteristic) phenomena to have a reducible indecomposable representation.

\section{The action of $G$ on $\HD$}\label{sec:HD}

As $\HD$ is functorial in $X$, it has a $k$-linear action of $G$. To determine the structure of of $\HD$ as a $k[G]$-module, we will use the following classical result about its structure as a $k$-vector space. Throughout the paper, we use the notation that $P_t= (t,0)$ for $t=0,1,\dots,p-1,$ and $P_{\infty}$ is the ``point at infinity" of $X$.

\begin{lemma}\label{lem:HDbasis}
Let 
\begin{equation}\label{HDbasis}
\tilde{\tau}_i = \frac{x^idx}{y}.
\end{equation}
The set $\{ \tilde{\tau}_i \mid i= 0, 1, \dots, g-1\}$ is then a $k$-basis for $\HD$.
\end{lemma}

\begin{proof} This follows from the calculations below, which we will use later in the paper:
\begin{align*}
\div(x) &= 2P_0 + (-2)P_{\infty}\\
\div(y) &= P_0 + P_1 + ... + P_{p-1} + (-p)P_{\infty}\\
\div(dx) &= P_0 + P_1 + ... + P_{p-1}+ (-3)P_{\infty}.
\end{align*}
So $\div(\tilde{\tau}_i)=2iP_0 + (p-2i-3)P_\infty$, and $\tilde{\tau}_i$ is regular for $i=0, \ldots, g-1$. They are linearly independent since they have different pole orders at $P_\infty$, and thus a $k$-basis.
\end{proof}

To study $\HD$ as a $G$-representation, we begin by considering a specific subgroup of small index. Consider the morphism
$\theta: \SL_2(\F_p) \rightarrow G$ that sends $\sigma \in \SL_2(\F_p)$ to the equivalence class of $(\sigma, 1)$, and let $H =\operatorname{im}\theta$.

One can check that
\begin{equation*}
\operatorname{ker}\theta =
\begin{cases}
I & \text{ if } p \equiv 1 \text{ mod } 4 \\
\pm I & \text{ if } p \equiv 3 \text{ mod } 4.
\end{cases}
\end{equation*}
As $|G|=2(p-1)p(p+1)$ and $|\SL_2(\F_p)|=(p-1)p(p+1)$, we conclude $|G:H|=2$ if $p\equiv 1$ mod 4, and that $|G:H|=4$ if $p\equiv 3$ mod 4.

Let $\Sym^{g-1}(k^2)$ be the $(g-1)^{st}$ symmetric power of the standard representation of $\SL_2(\F_p)$ on $k^2$. We identify $\Sym^{g-1}(k^2)$ with the $k$-subspace of $k[u,v]$ consisting of homogeneous polynomials with degree $g-1$. This space has basis $\{ v^{g-1}, v^{g-2}u, \dots , u^{g-1}\}$ and $\SL_2(\F_p)$-action given by \[
\left( \begin{array}{cc}
a & b\\
c & d
\end{array}
\right)
\left( \begin{array}{c}
u \\
v 
\end{array}
\right) =
\left( \begin{array}{c}
au+bv \\
cu+dv
\end{array}
\right).
\]
Note that if $p \equiv 3$ mod $4$, then $g-1$ is even, and the action of $\SL_2(\F_p)$ on $\Sym^{g-1}(k^2)$ is trivial on $\pm I$. Because of this, we can define an $H$-action on $\Sym^{g-1}(k^2)$ where if $h \in H$ such that $\theta(\sigma)=h$, then $h$ acts as $\sigma$ does. This is well-defined since the kernel of $\theta$ acts trivially. Let $V$ be this $H$-representation.

\begin{prop}\label{irrHom}
The map \begin{equation*}
\varphi: \operatorname{res}_H(\HD) \rightarrow V \qquad \text{ given by }\qquad \frac{x^idx}{y}\mapsto u^iv^{g-i-1}
\end{equation*} is a $k[H]$-module isomorphism.
\end{prop}

\begin{proof}
Clearly $\varphi$ is an isomorphism of vector spaces over $k$, so we need only check that $\varphi$ is $H$-equivariant. Let $h=(\gamma,1) \in H$ for $\gamma= \left(\begin{smallmatrix}
a & b\\
c & d\\
\end{smallmatrix}\right) \in \SL_2(\F_p)$. We calculate
\begin{align*}
h \circ \varphi\left(x^iy^{-1}dx\right) &= h(u^iv^{g-i-1})\\
&=(au+bv)^i(cu+dv)^{g-i-1}
\end{align*} and
\begin{align*}
\varphi \circ h(x^iy^{-1}dx) &= \varphi \left(\left(\frac{ax+b}{cx+d}\right)^i\left(\frac{y}{(cx+d)^{g+1}}\right)^{-1}d\left(\frac{ax+b}{cx+d}\right)\right)\\
&= \varphi((ax+b)^i(cx+d)^{g-i-1}y^{-1}dx)\\
&= (au+bv)^i(cu+dv)^{g-i-1},
\end{align*}
since $\displaystyle d\left(\frac{ax+b}{cx+d}\right) = (cx+d)^{-2} dx$.
\end{proof}

\begin{cor}\label{irreducible}
As a $G$-representation, $\HD$ is irreducible. Moreover, $\HF$ is naturally the contragredient of $\HD$, hence irreducible as well.
\end{cor}
\begin{proof} It is well-known that $\Sym^{g-1}(k^2)$ is irreducible as a representation of $\SL_2(\F_p)$ (see the discussion following Corollary 3 of Chapter 3 in \cite{Alp} 14-16). It is clear from the way we defined $V$ that it is thus irreducible as an $H$-representation, and this implies it is irreducible as a $G$-representation. The Serre duality pairing $\HD \times \HF \rightarrow k$ is $G$-equivariant, so $\HF$ is canonically the contragredient of $\HD$, and thus irreducible. \end{proof}

We can also explicitely describe the structure of $\HD$ as a subrepresentation of the induced representation of $V$.

\begin{prop}
If $p \equiv 1$ mod 4, then $\operatorname{Ind}_H^G(V) \cong V^2$ and $\HD$ can be embedded as the subrepresentation \[\{ (v,-v): v \in V\}.\] If $p \equiv 3$ mod 4, then $\operatorname{Ind}_H^G(V) \cong V^4$, then $\HD$ can be embedded as the subrepresentation spanned by \[ \{(\varphi(\tilde{\tau}_j), (-1)^ji \varphi(\tilde{\tau}_j), -\varphi(\tilde{\tau}_j), (-1)^{j+1}i\varphi(\tilde{\tau}_j)) \in V^4 \mid j=0, \dots, g-1\},\]  where $i \in k$ such that $i^2=-1$, and using the same notation as in Lemma~\ref{HDbasis} and Proposition~\ref{irrHom} for $\tilde{\tau}_j$ and $\varphi$.
\end{prop}

\begin{proof}Recall that the induced representation $\operatorname{Ind}_{H}^G(V)$ is $k[G] \otimes_{k[H]} V$. Frobenius reciprocity tells us that 
\[ \operatorname{Hom}_G(\HD, \operatorname{Ind}_H^G(V)) = \operatorname{Hom}_H(\operatorname{res}_H(\HD), V).\]
However, since Corollary~\ref{irreducible} tells us that $\operatorname{res}_H(\HD)$ and $V$ are isomorphic irreducible $H$-representations, this means that the right hand side is one dimensional. Thus there is a nonzero $k[G]$-linear map from $\HD$ into $\operatorname{Ind}_H^G(V)$ and it is unique up to scaling. (For more on irreducible and induced representations, see sections 1 and 8 respectively of \cite{Alp}.)

In the case where $p\equiv 1$ mod 4, $|G:H|=2$ so as a vector space the induced representation is $V \oplus V$. If $\alpha \in G$ corresponds to the equivalence class of $(\left( \begin{smallmatrix} 1 & 0 \\ 0 & 1\end{smallmatrix}\right), -1) \in \tilde{G}$, then $G = H \sqcup H \alpha$. Note in particular that if $\omega \in \HD$, the $\alpha \omega = -\omega$ and if $(v_1, v_2) \in \operatorname{Ind}_H^G(V)$, then $\alpha(v_1, v_2) = (v_2, v_1)$. Note that we can easily calculate the $G$-action on the induced representation from that.

In particular the map
\begin{align*}
T:\HD &\rightarrow \operatorname{Ind}_H^G(V)\\
\omega &\mapsto (\varphi(\omega), -\varphi(\omega))
\end{align*}
is $G$-equivariant. Thus $\HD$ embeds into $\operatorname{Ind}_H^G(V)$ as the $k[G]$-submodule $\{(v, -v)\in V \oplus V \}$.

If $p \equiv 3$ mod 4, then $|G:H|=4$ so as a vector space the induced representation is $V \oplus V \oplus V \oplus V$. If $\beta \in G$ corresponds to the equivalence class of $((\begin{smallmatrix} -1 & 0 \\ \phantom{-}0 & 1\end{smallmatrix}), i)$ where $i \in k$ such that $i^2=-1$, then $G = H \sqcup H \beta \sqcup H \beta^2 \sqcup H \beta^3$. Note in particular that if $\tilde{\tau}_j$ is a basis element of $\HD$ as in Lemma~\ref{HDbasis}, then $\beta\tilde{\tau}_j = (-1)^ji \tilde{\tau}_j$. If $(v_1, v_2, v_3, v_4) \in \operatorname{Ind}_H^G(V)$, then $\beta(v_1, v_2, v_3, v_4) = (\beta v_4, \beta v_1, \beta v_2, \beta v_3)$; note we can easily calculate the $G$-action on the induced representation from this. In particular, the injective map
\begin{align*}
T: \HD &\rightarrow \operatorname{Ind}_H^G(V)\\
\tilde{\tau}_j &\mapsto (\varphi(\tilde{\tau}_j), (-1)^ji \varphi(\tilde{\tau}_j), -\varphi(\tilde{\tau}_j), (-1)^{j+1}i\varphi(\tilde{\tau}_j))
\end{align*}
is $G$-equivariant.\end{proof}

\section{The action of $G$ on $\HdR$}

For any smooth proper curve $Y$ over a field $L$, $H^0(Y, \Omega_{Y/L})$ embeds as strict subrepresentation of $H_{dR}^1(Y/L)$. If $L$ is of characteristic zero, it follows from Maschke's Theorem that $H^0(Y, \Omega_{Y/L})$ is a $L[\Aut(Y)]$-module direct summand of $H_{dR}^1(Y/L)$. However, Maschke's Theorem does not apply in positive characteristic if the characteristic of $L$ divides $|\Aut(Y)|$ (as in our case), so it need not be the case that reducible implies decomposable. In this section we will show that, in fact, $\HdR$ is indecomposable as a $k[G]$-module.

For these computations we will use the \v{C}ech cohomology with the open affine cover $\mathcal{U}=\{U_1, U_2\}$, where $U_1=X-P_0$ and $U_2=X-P_{\infty}$ (for background on \v{C}ech hypercohomology, see EGA $0_{\text{III}}$ \S12 \cite{EGA} and SGA Exp V \cite{SGA}). The method used here is largely based on that in \cite{HaaJan}.  By definition, $\vHdR$ is the quotient of the $k$-vector space
\begin{equation*}
\{ (\omega_1, \omega_2, f_{12})\mid \omega_i \in \Omega_{X/k}(U_i), f_{12} \in \mathcal{O}_X(U_1 \cap U_2), df_{12}=\omega_1-\omega_2 \}
\end{equation*} by the $k$-subspace
\begin{equation*}
\{(df_1, df_2, f_1-f_2)\mid f_i \in \mathcal{O}_X(U_i) \}.
\end{equation*}

\begin{theorem}\label{nosplit}
The canonical exact sequence of $k[G]$-modules
\begin{equation} \label{eq:exactseq}
0 \rightarrow \HD \rightarrow \HdR \rightarrow \HF \rightarrow 0
\end{equation}
does not split.\end{theorem}

To prove this, we will first find a $k$-basis for $\vHdR$ and study the action of $G$ on this basis.

\begin{lemma}\label{lem:basisdR}
For $i=0,1,\dots, g-1$, let $\tau_i$ be the class of \begin{equation}(x^iy^{-1}dx, x^iy^{-1}dx, 0),
\end{equation} in $\vHdR$. For $i=1,2,\dots, g$, let $\eta_i$ be the class of \begin{equation}((1-2i)x^{1-g-i}d(yx^{-g-1}), -2ix^{2g-i}dy, yx^{-i})
\end{equation} in $\vHdR$. Together, these form a $k$-basis for $\vHdR$.
\end{lemma}
\begin{proof} We need to first determine that these are well-defined elements in $\vHdR$.
It is clear the $\tau_i$ are well defined in $\HdR$ since $x^iy^{-1}dx$ is regular for $i=0, 1, \dots, g-1$ (note: the $\tau_i$ are the images of the $\tilde{\tau}_i~\in~\check{H}^1(\mathcal{U}, \Omega^1_{X/k})$ under the canonical map). To show that $\eta_i$ is well defined, we first calculate that $d(yx^{-g-1})= x^{g-1}dy$. So
\begin{align*}
(1-2i)x^{1-g-i}d(yx^{-g-1}) - (-2ix^{2g-i}dy) &= (1-2i)x^{-i}dy+2ix^{2g-i}dy\\
&= d(yx^{-i}).
\end{align*}
The requirement that $\omega_1-\omega_2 = df_{12}$ for an element $(\omega_1, \omega_2, f_{12}) \in~\vHdR$ follows from this, but we still need to ensure that the $\omega_1$, $\omega_2$ and $f_{12}$ are regular on the appropriate open sets.  To do this, it is enough to calculate their divisors, which are:
\begin{align*}
\div((1-2i)x^{1-g-i}d(yx^{-g-1})) &= (-2i)P_0 + (2g+2i -2)P_{\infty}\\
\div(-2ix^{2g-i}dy) &= 2(2g-i)P_0 + (2i-2g-2)P_{\infty}\\
\div(yx^{-i}) &= (1-2i)P_0 + \sum_{j=1}^{p-1}P_j + (2i-p)P_{\infty}.
\end{align*}
Notice that the first equation refers to a 1-form regular on $U_1$, the second a 1-form regular on $U_2$, and the third a function regular on $U_1 \cap U_2$; we conclude that the $\eta_i$s are well-defined in $\vHdR$.

Since by Lemma \ref{lem:HDbasis}, the $\tilde{\tau}_i$s form a basis for $\check{H}^0(\mathcal{U}, \Omega^1_{X/k})$, and the $\tau_i$s are their images in $\vHdR$, it suffices to show that the image of $\{ \eta_i | i=1, \dots, g\}$ in $\check{H}^1(\mathcal{U}, \mathcal{O}_X)$ is a basis.

Recall that $\HF$ is the dual of $\HD$, and that under the canonical Serre-duality pairing there is a map $\HD \times \HF \rightarrow k$ given by the residue map (see \cite{Tate} or Appendix B of \cite{Con} for a complete discussion). We can use this to show that $\displaystyle -2yx^{-i}$ is the dual element to $y^{-1}x^{i-1}dx$. 

For $i=0, \dots, g-1$ and $j= 1, \dots, g$ the pairing gives
\begin{equation}
\langle y^{-1}x^idx, yx^{-j} \rangle = \operatorname{res}(x^{i-j}dx).
\end{equation}
This can be calculated by summing the residues at points $P$ where $P \in U_1$, or equivalently (by applying the Residue Theorem), the negative of the residue at $P_0$. Note in particular this means for the inner product to be non-zero, $x^{i-j}dx$ must have a pole somewhere on $U_1$ and a pole at $P_0$. By the calculations in Lemma~\ref{lem:HDbasis},
\begin{equation*}
\div(x^{i-j}dx) = (2i-2j+1)P_0 + P_0 + P_1 + \dots + P_{p-1} +(2j-2i- 3)P_\infty.
\end{equation*}
This has a pole at $P_0$ if $j \ge i+1$ and a pole at $P_\infty$ if $j \le i+1$. It follows that $\langle y^{-1}x^idx, yx^{-j} \rangle=0$ if $i \not = j-1$. If $i=j-1$, then $\langle y^{-1}x^idx, yx^{-j} \rangle=\operatorname{res}_{P_\infty}(x^{-1}dx)$. To calculate this residue, note that $\displaystyle t=\frac{y}{x^{(p+1)/2}}$ is a uniformizer at $P_\infty$ (using the proof of Lemma~\ref{lem:HDbasis}). Since \[ t^2 = \frac{y^2}{x^{p+1}} = \frac{x^p-x}{x^{p+1}} = \frac{1}{x} - \frac{1}{x^p}\] we find
\begin{equation*}
\frac{1}{x} = t^2 + \left(\frac{1}{x}\right)^p,
\end{equation*}
from which it follows (by repeatedly substituting) that $\displaystyle x^{-1} = \sum_{i=0}^\infty t^{2p^i}$. It is straightforward to show from this that $\operatorname{res}_{P_\infty}(x^{-1}dx)=-2$.

Thus $\{-2yx^{-i}\mid i=1, \dots, g\}$ is the dual basis of the $\{ \tilde{\tau}_{i-1}\mid i =1, \dots, g\}$, and in particular a basis for $\HF$.
\end{proof}

{\em Proof of Theorem~$\ref{nosplit}$.} Now that we have a basis for $\HdR$ as a split vector space, we need to consider how our group $G$ acts on these basis vectors.  Specifically, we are going to consider $\sigma \in G$, the equivalence class of  $\left((\begin{smallmatrix}
1 & 1\\
0 & 1\\
\end{smallmatrix}), 1\right)$, and its action on the span of $\{ \eta_i \mid i=1, \dots, g \}$, the subspace (non-canonically) isomorphic to $\HF$. We will calculate the image of each $\eta_i$ under $\sigma$, and use this to show that there is no splitting of the sequence as $k[G]$-modules.

Here we run into the problem that our cover $\mathcal{U}$ is not preserved under the group action.  While $\sigma U_2=U_2$, we get that $\sigma U_1= X - P_{p-1}$.  To account for this, we will refine our cover.  Let $U_3=X-P_1$.  Note that then $U_3=\sigma^{-1}U_1$.  Define the covers $\mathcal{U}'=\{U_2, U_3\}$ and $\mathcal{U}''=\mathcal{U} \cup \mathcal{U}'$. 

We can then calculate the action of $\sigma$ on $\vHdR$ using the following commutative diagram:

\centerline{
\begin{xy}
(0,3)*+{}="A";
(0,17)*+{}="B";
(10,0)*+{\HdR \cong \check{H}^1_{dR}(\mathcal{U})}="LD";
(10,20)*+{\HdR \cong \check{H}^1_{dR}(\mathcal{U})}="LU";
(60,0)*+{\check{H}^1_{dR}(\mathcal{U}')}="RD";
(60,20)*+{\check{H}^1_{dR}(\mathcal{U}'')}="RU";
{\ar@{->}^{\sigma} "A";"B"};
{\ar@{->}_{\qquad\cong}^{\qquad\rho} "RU";"LU"};
{\ar@{->}^{\rho'}_{\cong} "RU";"RD"};
{\ar@{->}^{\qquad\sigma} "LD";"RD"};
\end{xy}} 
\noindent where $\rho, \rho'$ are the restriction maps, which give isomorphisms of the spaces. As we did with $\check{H}^1_{dR}(\mathcal{U})$, we have a similar definition for $\check{H}^1_{dR}(\mathcal{U''})$, namely as the quotient of 
\begin{multline*} \{(\omega_1, \omega_2, \omega_3, f_{12}, f_{13}, f_{23}) \mid \omega_j \in \Omega_{X/k}(U_j), f_{jk} \in \mathcal{O}_X(U_j \cap U_k)\\ df_{jk}=\omega_j-\omega_k, f_{23}-f_{13}+f_{12}=0\},
\end{multline*}
by the subspace spanned by \[ \{(df_1, df_2, df_3, f_1-f_2, f_1-f_3, f_2-f_3) \mid f_j \in \mathcal{O}_X(U_j)\}.\] Then $\rho$ and $\rho'$ are respectively the projections to $(\omega_1, \omega_2, f_{12})$ and $(\omega_2, \omega_3, f_{23})$.

\begin{lemma}\label{lem:nubasisHdR}
For $i = 1, \dots, g$, let
\begin{align*}
\omega_{1i} &= (1-2i)x^{-g-i+1}d(yx^{-g-1})\\
\omega_{2i} &= -2ix^{p-i-1}dy\\
\omega_{3i} &= \sum_{m=1}^{p-i} \binom{p-i}{m} (1+2m)(x-1)^{m-3g} d(y(x-1)^{-g-1})\\
f_{12i} &= yx^{-i}\\
f_{23i} &= \sum_{m=1}^{p-i} \binom{p-i}{m}y(x-1)^{m-p}=\left(\frac{x^m-1}{x^p-1}\right)y\\
f_{13i} &= yx^{-i} +\left(\frac{x^m-1}{x^p-1}\right)y.
\end{align*}
Let $\nu_i =(\omega_{1i}, \omega_{2i}, \omega_{3i}, f_{12i}, f_{13i}, f_{23i})$. If $i=1, \dots, g$, then $\nu_i$ is a well-defined hyper 1-cocycle for the covering $\mathcal{U}''$. Moreover, under the canonical refinement map $\check{H}^1_{dR}(\mathcal{U}'') \isommap \check{H}^1_{dR}(\mathcal{U})$, the image of each $\nu_i$ is $\eta_i$.
\end{lemma}

\begin{proof} Since it is clear that the projection of $\nu_i$ to $\check{H}^1_{dR}(\mathcal{U})$ is $\eta_i$, it suffices to show that each $\nu_i$ is well-defined.

Since we know from Lemma \ref{lem:basisdR} that the $\eta_i$ are well-defined, we can conclude that $\omega_{1i} \in \Omega_{X/k}(U_1), \omega_{2i} \in \Omega_{X/k}(U_2)$, $f_{12i} \in \mathcal{O}_X(U_1 \cap U_2)$, and $df_{12i}= \omega_{1i}-\omega_{2i}$ for $i=1, \dots, g$. It then remains to verify the following conditions:
\begin{enumerate}
\item $\omega_{3i} \in \Omega_{X/k}(U_2)$
\item $f_{23i} \in \mathcal{O}_X(U_2~\cap~U_3)$
\item $f_{13i} \in \mathcal{O}_X(U_1~\cap~U_3)$
\item $df_{23i} = \omega_{2i} - \omega_{3i}$
\item $df_{13i} = \omega_{1i}-\omega_{3i}$
\item $f_{23i} - f_{13i} + f_{12i} = 0$
\end{enumerate}
Note that 6 is clear from the definition, and that with the other conditions it will imply 5. Next we will show that conditions 1, 2 and 4 on $f_{23i}$ and $\omega_{3i}$ follow from previous calculations. We know from Lemma~\ref{lem:basisdR}, that $(1-2i)x^{1-g-i}d(yx^{-g-1}) \in \Omega_{X/k}(U_1)$, $-2ix^{2g-i}dy \in \Omega_{X/k}(U_2)$, and that their difference is $d(yx^{-i}) \in \Omega_{X/k}(U_1 \cap U_2)$. Changing coordinates by replacing $x$ with $x-1$, we can conclude that 1-forms that previously were regular on $U_1$ will now be regular on $U_3$, while those regular on $U_2$ remain so, and that the equality still holds. Thus we get that
\begin{equation}\label{eq:calcbasis1}\displaystyle
(1-2i)(x-1)^{1-g-i}d(y(x-1)^{-g-1})+2i(x-1)^{2g-i}dy =d(y(x-1)^{-i})
\end{equation}
where \begin{align*}
(1-2i)(x-1)^{1-g-i}d(y(x-1)^{-g-1}) &\in \Omega_{X/k}(U_3)\\
-2i(x-1)^{2g-i}dy &\in \Omega_{X/k}(U_2)\\
y(x-1)^{-i} &\in \mathcal{O}_X(U_2~\cap~U_3).
\end{align*} If we substitute $\displaystyle i=p-m$ into (\ref{eq:calcbasis1}) (recalling that $p=2g+1$), it becomes
\begin{equation*}
(1+2m)(x-1)^{m-3g}d(y(x-1)^{-g-1}) - 2m (x-1)^{m-1}dy = d(y(x-1)^{m-p}).
\end{equation*}
Now take, multiply it by $\displaystyle \binom{p-i}{m}$, take the sum of these from $m=1$ to $m=p-i$. Referencing back to the definitions of $\omega_{3i}$ and $f_{23i}$, we find
\begin{equation}\label{eq:interm}
-\omega_{3i} + \sum_{m=1}^{p-1} \binom{p-1}{m} 2m (x-1)^{m-1}dy = df_{23i}.
\end{equation}
It follows from this that conditions 1 and 2 hold, that is $\omega_{3i} \in \Omega_{X/k}(U_2)$ and $f_{23i} \in \mathcal{O}_{X/k}(U_2 \cap U_3)$. Now note that
\begin{align*}
\omega_{2i} &= -2i x^{p-i-1}dy\\
&= -2i \sum_{\ell=0}^{p-i-1} \binom{p-i-1}{\ell} (x-1)^{\ell}dy\\
&= \sum_{m=1}^{p-i} \binom{p-i}{m} 2m (x-1)^{m-1}dy.
\end{align*}
So we can substitute this into (\ref{eq:interm}), and it follows that $\omega_{2i}-\omega_{3i}=df_{23}$, which is condition 4.

The only thing that remains is to check that $f_{13i} \in \mathcal{O}_X(U_1 \cap U_3)$. Since $f_{23i}$ and $f_{12i}$ are regular on $U_1 \cap U_2 \cap U_3$, we need only check that $f_{13i}$ is regular at $P_\infty$. We can do this by calculating the image of $f_{13i}$ in $\operatorname{Frac}(\mathcal{O}_{X,P_\infty})$. Recall from the proof of Lemma~\ref{lem:basisdR}, that $\displaystyle t = \frac{y}{x^{(p+1)/2}}$ is a uniformizer at $P_\infty$. Using this and the definition of $f_{13i}$, we calculate that $f_{13i}=O(t)$. Since this means $f_{13i}$ has no pole at $P_\infty$, we can conclude $f_{13i}$ is regular on $U_1 \cap U_3$. So $f_{13i} \in \mathcal{O}_X(U_1 \cap U_3)$.
\end{proof}

Now we conclude the proof of Theorem \ref{nosplit}. By Lemma \ref{lem:nubasisHdR}, each $\nu_i$ is the inverse image of $\eta_i$ under the canonical restriction map $\check{H}_{dR}^1(\mathcal{U}'') \isommap \check{H}_{dR}^1(\mathcal{U})$. The projection of $\nu_i$ onto $\check{H}_{dR}^1(\mathcal{U}')$ gives us $(\omega_{2i}, \omega_{3i}, f_{23i})$, and to understand $\sigma \eta_i$, we need only calculate $\sigma(\omega_{2i}, \omega_{3i}, f_{23i})$.

So we need to consider the image of $(\omega_{3i}, \omega_{2i}, f_{23i})$ under $\sigma$:
\begin{align*}
\sigma \omega_{3i} &= -\frac{i}{p-i}\sum_{j=1}^{p-i}\binom{p-i}{j}(1+2j)x^{j-3g}d(yx^{-g-1})\\
\sigma \omega_{2i} &= -2i(x+1)^{p-i-1}dy = -\frac{i}{p-i}\sum_{j=1}^{p-i}\binom{p-i}{j}2jx^{j-1}dy\\
\sigma f_{23i} &= -\frac{i}{p-i}\sum_{j=1}^{p-i}\binom{p-i}{j}yx^{j-p}.
\end{align*}
We can restate this as
\begin{equation}\label{eq:sigmadoes}
\sigma(\omega_{3i}, \omega_{2i}, f_{23i}) = -\frac{i}{p-i}\sum_{j=1}^{p-i}\binom{p-i}{j} \tilde{\eta}_{p-j}
\end{equation}
where
\begin{equation*}
\tilde{\eta}_\ell = ((1-2\ell)x^{1-g-i}d(yx^{-g-1}), -2\ell x^{2g-\ell}dy, yx^{-\ell}).
\end{equation*}
We would like to be able to write this in terms of the basis from Lemma~\ref{lem:basisdR}. If $g+1 \le j \le p-1$, then each $\tilde{\eta}_{p-j} = \eta_{p-j}$, which is such a basis element. However, for $1 \le j \le g$, this is not the case and so we need to rewrite $\tilde{\eta}_{p-j}$ in terms of the basis. We find that for $1 \le j \le g$,
\begin{equation*}
\tilde{\eta}_{p-j}-(d(yx^{j-p}), 0, yx^{j-p})=(2j x^{j-1}dy, 2j x^{j-1}dy, 0) = -2j\tau_{j-1}.
\end{equation*}
Since the above is in the same equivalence class as $\tilde{\eta}_{p-j}$ within $\check{H}_{dR}^1(\mathcal{U})$, for $1 \le j \le g$, $\tilde{\eta}_{p-j}$ is in the image of $\HD$.  This shows that the space spanned by the $\eta_i$ is not stable under the action of $\sigma$.

Suppose there is a $k[G]$-linear map $\alpha: \HF \rightarrow \HdR$ that splits the exact sequence in Equation (\ref{eq:exactseq}). Let $\{ \tilde{\tau_\ell}^* \in \HF \mid \ell=0, \dots, g-1 \}$ be a dual basis of $\{\tilde{\tau_\ell} \in \HD \mid \ell=0, \dots, g-1 \}$ under the Serre duality pairing. Then define $\tau_\ell^* = \alpha(\tilde{\tau_\ell}^*) \in \HdR$. We would like to know how to write $\tau_\ell^*$ in terms of the $\tau_i$ and $\eta_i$, the basis given in Lemma~\ref{lem:basisdR}.

Fix $t \in \F_p^\times$. and let $T$ indicate the element of $G$ such that \[(x,y) \mapsto (t^2x, ty).\]
It is easy to calculate that $T\tau_i = t^{2i+1}\tau_i$ and $T\eta_{j} = t^{1-2j}\eta_j$. Since the pairing is $G$-equivariant, as is the map $\alpha$, it has to be the case that $T\tau_\ell^* = t^{-2\ell-1}\tau_\ell^*$. This implies that $\tau_\ell^*$ is in the $t^{-2\ell-1}$-eigenspace of $T$, which is spanned by $\tau_{g-\ell-1}$ and $\eta_{\ell+1}$. So there exist unique nontrivial $a_\ell, b_\ell \in k$ such that $\tau_\ell^* = a_\ell \tau_{g-\ell-1} + b_\ell \eta_{\ell+1}$.

Since $\alpha$ is $G$-equivariant, $\sigma \tau_\ell^* = \sigma \alpha(\tilde{\tau_\ell}^*) =\alpha(\sigma \tilde{\tau_{\ell}}^*)$, so it must be that
\begin{equation}\label{asbs} a_\ell (\sigma \tau_{g-\ell-1}) + b_\ell (\sigma \eta_{\ell+1}) = \sigma \tau_\ell^* \in \operatorname{im} \alpha.\end{equation} However, we know from the proof of Theorem~\ref{irrHom} and Equation~\ref{eq:sigmadoes} that
\[ \sigma \tau_{g-\ell-1} = \sum_{i=0}^{g-\ell-1} \binom{g-\ell-1}{i} \tau_i\]
and
\[\sigma \eta_{\ell+1} = -\frac{\ell+1}{p-\ell-1}\left(-2 \sum_{i=1}^{g}\binom{p-\ell-1}{i}i \tau_{i-1}+\sum_{i=g+1}^{p-\ell-1} \binom{p-\ell-1}{i} \eta_{p-i}\right).\]

It is simple to show, by plugging this into Equation (\ref{asbs}), that no nontrivial $a_\ell, b_\ell$ satisfy this, giving a contradiction. Thus the exact sequence in Equation (\ref{eq:exactseq}) does not split under the action of $G$. \qed

\begin{cor}
$\HdR$ is a non-projective indecomposable $k[G]$-module.
\end{cor}
\begin{proof}
Corollary 7 on page 33 of \cite{Alp} states that the order of a Sylow $p$-group of $G$ divides the dimension of any projective $k[G]$-module. A Sylow $p$-group of $G$ has order $p$ and the dimension of $\HdR$ is $2g=p-1$. So $\HdR$ is non-projective.

Suppose there exist nontrivial $k[G]$-submodules $M$ and $N$ such that \[\HdR = M \oplus N. \] 
Consider the sequence from Theorem~\ref{nosplit}:
\begin{equation}\label{eq:unsplittable}
0 \rightarrow \HD \xrightarrow{\psi} \HdR \xrightarrow{\psi'} \HF \rightarrow 0 \end{equation}
where $\psi, \psi'$ are the canonical maps. Suppose that $\operatorname{im} \psi \cap M=0$. Then $\operatorname{ker} \psi' \cap M = 0$, so $\psi'|M$ gives an isomorphism between $M$ and some submodule of $\HF$. However, since $\HF$ is irreducible by Corollary~\ref{irreducible}, this means that $M$ is isomorphic to $\HF$ as a $k[G]$-module. However, this gives a splitting of the exact sequence in (\ref{eq:unsplittable}), which contradictions Theorem~\ref{nosplit}.

Now suppose instead that $\operatorname{im} \psi \cap M$ is nonzero. Since $\operatorname{im} \psi \cong \HD$ as a $k[G]$-module, and $\HD$ is irreducible by Corollary~\ref{irreducible}, we can conclude that $\operatorname{im} \psi \subseteq M$. So $\operatorname{ker} \psi' \subseteq M$, and thus $\operatorname{ker} \psi' \cap N$ is zero. Thus, $N$ is isomorphic to a $k[G]$-submodule of $\HF$ through $\psi'$. Since $N$ is nontrivial and $\HF$ is irreducible by Corollary~\ref{irreducible}, this means that $\psi'$ induces an isomorphism between $N$ and $\HF$. However, this gives a splitting of the exact sequence in (\ref{eq:unsplittable}), which contradicts Theorem \ref{nosplit}.
\end{proof}

\bibliographystyle{plain}
\bibliography{Hortschrefs}

\begin{thebibliography}{10}

\bibitem{Alp}
J.~L. Alperin.
\newblock {\em Local representation theory}, volume~11 of {\em Cambridge
  Studies in Advanced Mathematics}.
\newblock Cambridge University Press, Cambridge, 1986.
\newblock Modular representations as an introduction to the local
  representation theory of finite groups.

\bibitem{SGA}
M.~Artin, J.~E. Bertin, M.~Demazure, P.~Gabriel, A.~Grothendieck, M.~Raynaud,
  and J.-P. Serre.
\newblock {\em Sch\'emas en groupes. {F}asc. 2a: {E}xpos\'es 5 et 6}, volume
  1963/64 of {\em S\'eminaire de G\'eom\'etrie Alg\'ebrique de l'Institut des
  Hautes \'Etudes Scientifiques}.
\newblock Institut des Hautes \'Etudes Scientifiques, Paris, 1963/1965.

\bibitem{Bre}
Thomas Breuer.
\newblock {\em Characters and automorphism groups of compact {R}iemann
  surfaces}, volume 280 of {\em London Mathematical Society Lecture Note
  Series}.
\newblock Cambridge University Press, Cambridge, 2000.

\bibitem{ChWe}
C.~Chevalley, A.~Weil, and E.~Hecke.
\newblock \"{U}ber das {V}erhalten der {I}ntegrale 1. {G}attung bei
  {A}utomorphismen des {F}unktionenk\"orpers.
\newblock {\em Abhandlungen aus dem Mathematischen Seminar der Universit\"at
  Hamburg}, 10:358--361, 1934.
\newblock 10.1007/BF02940687.

\bibitem{Con}
Brian Conrad.
\newblock {\em Grothendieck duality and base change}, volume 1750 of {\em
  Lecture Notes in Mathematics}.
\newblock Springer-Verlag, Berlin, 2000.

\bibitem{Elk}
Noam~D. Elkies.
\newblock The {K}lein quartic in number theory.
\newblock In {\em The eightfold way}, volume~35 of {\em Math. Sci. Res. Inst.
  Publ.}, pages 51--101. Cambridge Univ. Press, Cambridge, 1999.

\bibitem{GJK}
Darren Glass, David Joyner, and Amy Ksir.
\newblock Codes from {R}iemann-{R}och spaces for {$y^2=x^p-x$} over {${\rm
  GF}(p)$}.
\newblock {\em Int. J. Inf. Coding Theory}, 1(3):298--312, 2010.

\bibitem{EGA}
A.~Grothendieck.
\newblock \'{E}l\'ements de g\'eom\'etrie alg\'ebrique. {III}. \'{E}tude
  cohomologique des faisceaux coh\'erents. {I}.
\newblock {\em Inst. Hautes \'Etudes Sci. Publ. Math.}, (11):167, 1961.

\bibitem{HaaJan}
Burkhard Haastert and Jens~Carsten Jantzen.
\newblock Filtrations of the discrete series of {${\rm SL}_2(q)$} via
  crystalline cohomology.
\newblock {\em J. Algebra}, 132(1):77--103, 1990.

\bibitem{Kani}
Ernst Kani.
\newblock The {G}alois-module structure of the space of holomorphic
  differentials of a curve.
\newblock {\em J. Reine Angew. Math.}, 367:187--206, 1986.

\bibitem{Katz}
Nicholas~M. Katz.
\newblock Crystalline cohomology, {D}ieudonn\'e modules, and {J}acobi sums.
\newblock In {\em Automorphic forms, representation theory and arithmetic
  ({B}ombay, 1979)}, volume~10 of {\em Tata Inst. Fund. Res. Studies in Math.},
  pages 165--246. Tata Inst. Fundamental Res., Bombay, 1981.

\bibitem{Ko}
Bernhard K{\"o}ck.
\newblock Galois structure of {Z}ariski cohomology for weakly ramified covers
  of curves.
\newblock {\em Amer. J. Math.}, 126(5):1085--1107, 2004.

\bibitem{Nak1}
S.~Nakajima.
\newblock On {G}alois module structure of the cohomology groups of an algebraic
  variety.
\newblock {\em Invent. Math.}, 75(1):1--8, 1984.

\bibitem{Nak3}
Sh{\=o}ichi Nakajima.
\newblock Action of an automorphism of order {$p$} on cohomology groups of an
  algebraic curve.
\newblock {\em J. Pure Appl. Algebra}, 42(1):85--94, 1986.

\bibitem{Nak2}
Sh{\=o}ichi Nakajima.
\newblock Galois module structure of cohomology groups for tamely ramified
  coverings of algebraic varieties.
\newblock {\em J. Number Theory}, 22(1):115--123, 1986.

\bibitem{Roq}
Peter Roquette.
\newblock Absch\"atzung der {A}utomorphismenanzahl von {F}unktionenk\"orpern
  bei {P}rimzahlcharakteristik.
\newblock {\em Math. Z.}, 117:157--163, 1970.

\bibitem{Tate}
John Tate.
\newblock Residues of differentials on curves.
\newblock {\em Ann. Sci. \'Ecole Norm. Sup. (4)}, 1:149--159, 1968.

\bibitem{ValMad}
Robert~C. Valentini and Manohar~L. Madan.
\newblock Automorphisms and holomorphic differentials in characteristic {$p$}.
\newblock {\em J. Number Theory}, 13(1):106--115, 1981.

\bibitem{Weil}
Andr\'e Weil.
\newblock \"{U}ber {M}atrizenringe auf {R}iemannschen {F}l\"achen und den
  {R}iemann - {R}ochsehen {S}atz.
\newblock {\em Abhandlungen aus dem Mathematischen Seminar der Universit\"at
  Hamburg}, 11:110--115, 1935.
\newblock 10.1007/BF02940718.

\end{thebibliography}

\end{document}